\documentclass[]{tPHM2e}
\usepackage{amsmath, amssymb, bbm, graphs}

\def\G{\Gamma}
\def\S{\Sigma}

\newcommand{\set}[1]{\left\{#1\right\}}

\begin{document}
\doi{10.1080/1478643YYxxxxxxxx}
\issn{1478-6443}
\issnp{1478-6435}
\jvol{00} \jnum{00} \jyear{2008} \jmonth{21 December}

\markboth{M.J.C. Loquias and P. Zeiner }{Philosophical Magazine}

\title{Colourings of lattices and coincidence site lattices}

\author{Manuel Joseph C. Loquias$^{\rm a,b}$$^{\ast}$\thanks{$^\ast$Corresponding author. Email: mloquias@math.uni-bielefeld.de
\vspace{6pt}} and Peter Zeiner$^{\rm a}$\\
\vspace{6pt}  $^{\rm a}${\em{Faculty of Mathematics, Bielefeld University, Universit\"{a}tstra\ss e 25, 33615 Bielefeld, Germany}}; 
$^{\rm b}${\em{Institute of Mathematics, University of the Philippines, Diliman, C.P. Garcia St., 1101 Diliman, Quezon City, Philippines}}}

\maketitle

\begin{abstract}
	The relationship between the coincidence indices of a lattice $\G_1$ and a sublattice $\G_2$ of $\G_1$ is examined via the colouring of	$\Gamma_1$ that is obtained by assigning a unique 
	colour to each coset of $\G_2$.  In addition, the idea of colour symmetry, originally defined for symmetries of lattices, is extended to coincidence isometries of lattices.  An example 
	involving the Ammann-Beenker tiling is provided to illustrate the results in the quasicrystal setting.  
\end{abstract}

\section{Introduction}
	Coincidence site lattices are used in studying and describing the structure of grain boundaries in crystals.\cite{B70}  The discovery of quasicrystals led to a more mathematical study 
	of coincidence site lattices, including coincidence site lattices in dimensions $d\geq 4$ and their generalizations to
	modules.\cite{PBR96,B97,Z05,Z06,BZ08,BGHZ08,Zo06,H09}

	Similar to coincidence site lattices, a revived interest on colour symmetries in recent years was brought upon by Schechtman's discovery, see 
	\cite{MP94,L97,B97B,BGS02,BGS04,BG04,DLPF07,L08,LF08,BDLPEFF08}. Despite being two different problems, the enumeration and classification of colour symmetries of lattices and the 
	identification of coincidence site lattices come hand in hand.\cite{Sc80,Sc84,PBR96,B97B,L08,LF08}.  

	It is well-known that a lattice and its sublattices have the same set of coincidence isometries.  In fact, not only are the coincidence isometries for the three cubic lattices the 
	same, but also are the coincidence indices for a given coincidence isometry.  This is due to a particular shell structure of the lattice points: points at the corners of a cube and 
	those in the centre or on the faces of the cube lie on different shells, and thus cannot be mapped onto each other by an isometry.\cite{GBW74} On the other hand, aside from having 
	different point groups, the two four-dimensional hypercubic lattices do not share the same set of coincidence indices (the coincidence indices of the centred type are all odd while 
	those of the primitive type are odd or twice an odd number). \cite{B97}  This motivates the following questions: Under which conditions does a sublattice have the same coincidence 
	indices as the lattice itself? Is it possible to calculate the coincidence indices of a sublattice once the coincidence indices of its parent lattice have already been determined? In 
	this paper, answers to these questions are discussed by looking at certain colourings of lattices.

	A discrete subset $\G$ of $\mathbbm{R}^d$ is a $\emph{lattice}$ (of rank and dimension $d$) if it is the $\mathbbm{Z}$-span of $d$ linearly independent vectors over 
	$\mathbbm{R}$.  As a group, $\G$ is isomorphic to the free abelian group of rank $d$.  A \emph{sublattice} of $\G$ is a subgroup of $\G$ of finite index. 

	An $R\in O(d)$ is a \emph{coincidence isometry} of $\G$ if $\G(R):=\G\cap R\G$ is a sublattice of $\G$.  The sublattice $\G(R)$ is called a \emph{coincidence site lattice} 
	(CSL) while the index of $\G(R)$ in $\G$, $\S(R):=[\G:\G(R)]=[R\G:\G(R)]$, is called the \emph{coincidence index of} $R$.  The set of coincidence isometries of $\G$ forms a 
	group, denoted by $OC(\G)$.
	
	If $\G_2$ is a sublattice of $\G_1$ of index $m$, then $OC(\G_1)=OC(\G_2)$.\cite{B97}  However, the coincidence indices and corresponding multiplicities of $\G_1$ and $\G_2$ are 
	in general different.  What is known is if $R\in OC(\G_1)$, then $\S_1(R)\mid m\,\S_2(R)$ (where $\S_i(R)$ is the coincidence index of $R$ with respect to  $\G_i$, for 
	$i=1,2$).\cite{B97}  We shall derive a formula for $\S_2(R)$ in terms of $\S_1(R)$ by considering properties of the colouring of $\G_1$ that is induced by $\G_2$.
	
	A \emph{colouring} of $\G_1$ determined by the sublattice $\G_2$ of index $m$ is a bijection between the set of cosets of $\G_2$ in $\G_1$ and a 
	set $C$ of $m$ colours. For simplicity, we shall say that the coset $c_i+\G_2\subseteq \G_1$ has \emph{colour} $c_i$.

	Denote by $G$ the symmetry group of $\G_1$.  A symmetry in $G$ is called a \emph{colour symmetry} of the colouring of $\G_1$ induced by $\G_2$ if it permutes the colours in 
	the colouring, that is, all and only those points having the same colour are mapped by the symmetry to a certain colour.  The group formed by the colour symmetries of the 
	colouring of $\G_1$ is called the \emph{colour group} or \emph{colour symmetry group}, usually denoted by $H$ (\cite{DLPF07,Sc84}, see also \cite{GS87} for an equivalent 
	formulation).  Since $H$ acts on $C$,  there exists a homomorphism $f:H\rightarrow P(C)$ where $P(C)$ is the group of permutations of $C$.  The kernel, $K$, of $f$ is the 
	subgroup of $H$ whose elements fix the colours, in other words, $K$ is the symmetry group of the coloured lattice.  We obtain that the group of colour permutations of the 
	lattice, $f(H)$, is isomorphic to $H/K$.\cite{DLPF07}  

	We shall generalise the notion of a colour symmetry to that of a colour coincidence.  In addition, we associate the property that a coincidence isometry $R$ of $\G_1$ is a 
	colour coincidence of the colouring of $\G_1$ determined by a sublattice $\G_2$ with the relationship between the coincidence indices of $R$ with respect to $\G_1$ and $\G_2$.

\section{Coincidence index with respect to a sublattice}
	Let $\G_1$ be a lattice in $\mathbbm{R}^d$ and $\G_2$ be a sublattice of $\G_1$ of index $m$.  We write $\G_1=\bigcup_{i=1}^m(c_i+\G_2)$ where $c_1=0$, and we consider the colouring of 
	$\G_1$ determined by $\G_2$.  Fix an $R\in OC(\G_1)=OC(\G_2)$ and denote by $\S_i(R)$ the coincidence index of $R$ with respect to $\G_i$, for $i=1,2$.  Let 
	$I:=\set{c_i+\G_2:(c_i+\G_2)\cap\G_1(R^{-1})\neq \emptyset}$ and $J:=\set{c_j+\G_2:(c_j+\G_2)\cap\G_1(R)\neq\emptyset}$.  Without loss of generality, we assume throughout the paper that 
	$c_i\in\G_1(R^{-1})$ and $c_j\in\G_1(R)$ whenever $c_i+\G_2\in I$ and $c_j+\G_2\in J$.  We then obtain partitions of $\G_1(R^{-1})$ and $\G_1(R)$ that are induced by $\G_2$, given by 
	$\G_1(R^{-1})=\bigcup_{c_i+\G_2\in I}{(c_i+\G_2)\cap\G_1(R^{-1})}$ and $\G_1(R)=\bigcup_{c_j+\G_2\in J}{(c_j+\G_2)\cap\G_1(R)}$.
	These partitions correspond to colourings of $\G_1(R^{-1})$ and $\G_1(R)$, respectively, wherein the colours are inherited from the colouring of $\G_1$ determined by $\G_2$.  Observe 
	that all points of $\G_1(R^{-1})$ are mapped bijectively by $R$ to points of $\G_1(R)$, that is, $R[\G_1(R^{-1})]=\G_1(R)$.

	The following lemma tells us that $\G_2(R)$ consists of those points coloured $c_1$ in the colouring of $\G_1(R)$ whose preimages under $R$ are also points coloured $c_1$ in the 
	colouring of $\G_1(R^{-1})$.

	\begin{lemma}\label{intandeqthm}
		Let $\G_2$ be a sublattice of $\G_1$ and $R\in OC(\G_1)$.  Then 

		\makebox[\linewidth][c]{$[\G_2\cap\G_1(R)]\cap R[\G_2\cap\G_1(R^{-1})]=[\G_2\cap\G_1(R)]\cap[R\G_2\cap\G_1(R)]=\G_2(R)$.}

		\noindent In particular, if $R\G_2\cap\G_1(R)=\G_2\cap\G_1(R)$ then $\G_2(R)=R\G_2\cap\G_1(R)=\G_2\cap\G_1(R)$.
	\end{lemma}	
	\begin{proof}
		We have $[\G_2\cap\G_1(R)]\cap[R\G_2\cap\G_1(R)] =(\G_2\cap R\G_2)\cap\G_1(R)=\G_2(R)$.
	\end{proof}

	\begin{figure}
		\begin{center}
		 	\input{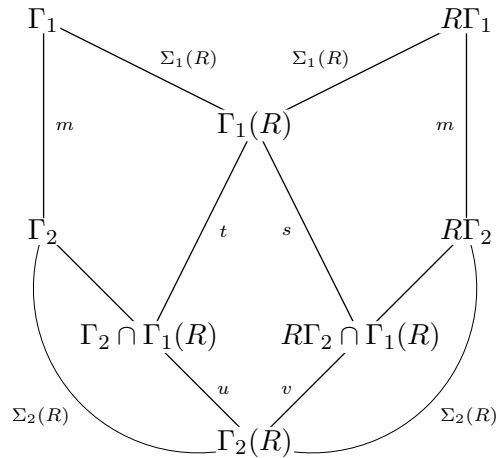}
		\end{center}
		\caption{\label{latticediag}Lattice diagram of the lattices $\G_1$, $R\G_1$, $\G_2$, $R\G_2$, $\G_1(R)$, and $\G_2(R)$ (as groups) and corresponding indices}
	\end{figure}

	Let $s:=[\G_1(R):R\G_2\cap\G_1(R)]$, $t:=[\G_1(R):\G_2\cap\G_1(R)]$, $u:=[\G_2\cap\G_1(R):\G_2(R)]$, and $v:=[R\G_2\cap\G_1(R):\G_2(R)]$ (see Figure \ref{latticediag}).  By comparing 
	indices in Figure \ref{latticediag},  a formula for $\S_2(R)$ in terms of $\S_1(R)$ is readily obtained.  In addition, restrictions on the values of $s$, $t$, $u$, and $v$, as well as  
	interpretations of these values in relation to the colourings of $\G_1(R^{-1})$ and $\G_1(R)$ determined by $\G_2$, can be deduced through repeated applications of the second 
	isomorphism theorem (for groups).  These results are stated in the following theorem, whose proof will be published elsewhere. \cite{LZp}

	\begin{theorem}\label{colouring}
		Consider the colouring of a lattice $\G_1$ determined by a sublattice $\G_2$ of $\G_1$ of index $m$ where each coset $c_i+\G_2$ is assigned the colour $c_i$ (for $i=1,\ldots, 
		m$) with $c_1=0$.  Let $R\in OC(\G_1)$ and $\S_i(R)$ be the coincidence index of $R$ with respect to $\G_i$, for $i=1,2$.  Then
		\begin{equation}\label{s2r}
		 	\S_2(R)=\frac{t\cdot u\cdot\S_1(R)}{m}=\frac{s\cdot v\cdot\S_1(R)}{m}
		\end{equation}
		where $s$ and $t$ are the number of colours in the colouring of $\G_1(R^{-1})$ and $\G_1(R)$, respectively, determined by $\G_2$; $u$ is the number of colours $c_i$ with the 
		property that some points of $\G_1(R^{-1})$ coloured $c_i$ are mapped by $R$ to points coloured $c_1$ in the colouring of $\G_1(R)$; and $v$ is the number of colours in the 
		colouring of $\G_1(R)$ that are intersected by the images under $R$ of those points of $\G_1(R^{-1})$ coloured $c_1$.  Moreover, $s|m$, $t|m$, $u|s$, and $v|t$.
	\end{theorem}

	An immediate consequence of Theorem \ref{colouring} is the well-known bound on the value of $\S_2(R)$: $\S_1(R)\mid m\,\S_2(R)$ and $\S_2(R)\mid m\,\S_1(R)$.\cite{B97}  Observe that 
	$|I|=s$, $|J|=t$, and $s=t$ if and only if $u=v$.

	\begin{example}\label{ex1}
		Consider the square lattice $\G_1=\mathbbm{Z}^2$ and the sublattice $\G_2$ of index 6 in $\G_1$ generated by $[6,0]^T$ and $[2,1]^T$.  Take the coincidence rotation $R$ of 
		$\G_1$ corresponding to a rotation about the origin by $\tan^{-1}(\frac{3}{4})\approx 37^{\circ}$ in the counterclockwise direction.  It is known that $\S_1(R)=5$.  Choose the 
		coset representatives $c_i=[i-1,0]^T$ for $i\in\set{1,\ldots,6}$ so that $\G_1=\bigcup_{i=1}^6(c_i+\G_2)$.  The colourings of $\G_1$, $\G_1(R^{-1})$, and $\G_1(R)$ determined by 
		$\G_2$ are shown in Figures \ref{colour}, \ref{g1rinv}, and \ref{g1r} which were obtained by designating the colours black, yellow, blue, red, gray, and green to the cosets 
		$c_1+\G_2$, $c_2+\G_2$, $c_3+\G_2$, $c_4+\G_2$, $c_5+\G_2$, and $c_6+\G_2$, respectively.  

		Since all six colours appear in both colourings of $\G_1(R^{-1})$ and $\G_1(R)$, $s=t=6$.  Observe that half of the black points in the colouring of $\G_1(R^{-1})$ are sent by 
		$R$ again to black points in the colouring of $\G_1(R)$ while the rest are sent to red points.  This implies that $u=2$.  Therefore, by \eqref{s2r}, $\S_2(R)=\frac{6\cdot 2\cdot 
		5}{6}=10$.  By Lemma \ref{intandeqthm}, $\G_2(R)$ is obtained by taking the intersection of $\G_2\cap\G_1(R)$ (the black points in the colouring of $\G_1(R)$) and 
		$R[\G_2\cap\G_1(R^{-1})]$ (the resulting points when the black points in the colouring of $\G_1(R^{-1})$ are rotated by $R$), and is shown in Figure \ref{g2r}.
		\begin{figure}
			\begin{center}	
			 	\input{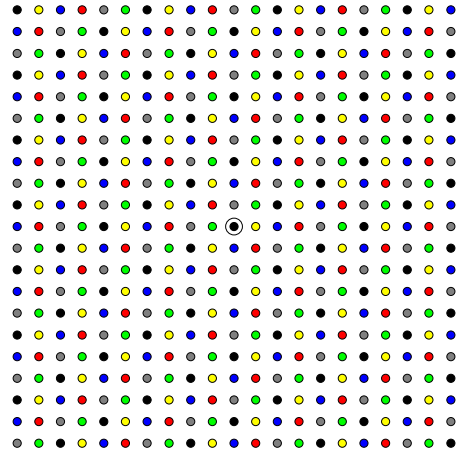}
			\end{center}
			\caption{Colouring of $\G_1=\mathbbm{Z}^2$ determined by the sublattice $\G_2$ ($=$ black dots) with basis $\set{[6,0]^T,[2,1]^T}$.  The encircled dot indicates the 
			origin. (Online: coloured version)}\label{colour}
		\end{figure}		
		\begin{figure}
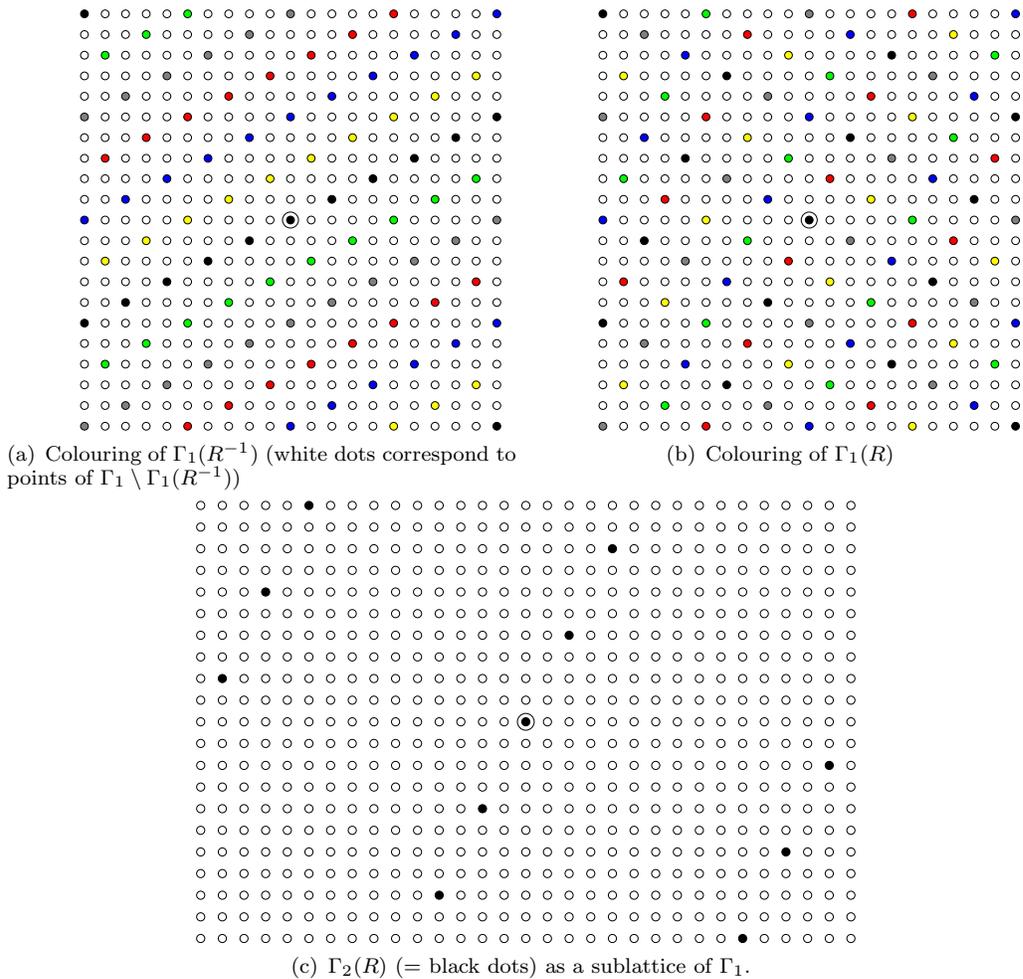

			\begin{center}	
				\subfigure[Colouring of $\G_1(R^{-1})$ (white dots correspond to points of $\G_1\setminus\G_1(R^{-1})$)]{\label{g1rinv}
				\input{exoneinvb.tex}				
				}
				\subfigure[Colouring of $\G_1(R)$]{\label{g1r}
				\input{exoneb.tex}
				}
				\subfigure[$\G_2(R)$ ($=$ black dots) as a sublattice of $\G_1$.]{\label{g2r}
				\input{exoneg2.tex}
				}
				\caption{Colourings of $\G_1(R^{-1})$ and $\G_1(R)$ determined by the sublattice $\G_2$ of Figure \ref{colour}, where $R$ is 
				the counterclockwise rotation about the origin by $\tan^{-1}(\frac{3}{4})\approx 37^{\circ}$.  The CSL $\G_2(R)$ of index 10 in $\G_2$ is obtained by taking all 
				black points of $\G_1(R)$ whose preimages under $R$ are also coloured black. (Online: coloured version)}\label{exone}
			\end{center}			
		\end{figure}

	\end{example}

\section{Colour coincidence}
	We have seen from the previous section how the interaction of the colours in the colourings of $\G_1(R^{-1})$ and $\G_1(R)$ determined by $\G_2$ affects $\G_2(R)$ and thus the value of 
	$\S_2(R)$.  This motivates the following definition.

 	\begin{definition}
		A coincidence isometry $R$ of the lattice $\G_1$ is called a \emph{colour coincidence} of the colouring of $\G_1$ determined by a sublattice $\G_2$ of $\G_1$ if $R$ defines a 
		bijection between the partitions $\set{(c_i+\G_2)\cap \G_1(R^{-1}):c_i+\G_2\in I}$ of $\G_1(R^{-1})$ and $\set{(c_j+\G_2)\cap \G_1(R):c_j+\G_2\in J}$ of $\G_1(R)$ given by 
		$R[(c_i+\G_2)\cap\G_1(R^{-1})]=(c_j+\G_2)\cap\G_1(R)$.	
 	\end{definition}

	In the above definition, the mapping between $(c_i+\G_2)\cap\G_1(R^{-1})$ and $(c_j+\G_2)\cap\G_1(R)$ means that all points, and only those points, coloured $c_i$ in the colouring of 
	$\G_1(R^{-1})$ are sent by $R$ to points coloured $c_j$ in the colouring of $\G_1(R)$.  In such a case, $R$ is said to send or map colour $c_i$ to colour $c_j$.  Furthermore, if $R$ 
	maps a colour $c_i$ onto itself, $R$ is said to \emph{fix the colour $c_i$}.  Hence, a colour coincidence $R$ determines a bijection between the set of colours in the colouring of 
	$\G_1(R^{-1})$ and the set of colours in the colouring of $\G_1(R)$.  In particular, if a colour coincidence $R$ is an element of the symmetry group of $\G_1$ (that is, 
	$\G_1(R^{-1})=\G_1(R)=\G_1$), then $R$ is a colour symmetry of the colouring of $\G_1$. 
	
	Clearly, the identity in $OC(\G_1)$ is a colour coincidence of any colouring of $\G_1$.  In addition, if $R$ is a colour coincidence of a colouring of $\G_1$, then so is $R^{-1}$.  The 
	question of whether the set of colour coincidences of a colouring of a lattice forms a group is yet to be resolved.

	The next theorem, which is a generalisation of the first result in Theorem 2 of \cite{DLPF07} on colour symmetries of colourings of square and hexagonal lattices, provides a 
	characterisation of colour coincidences of lattice colourings determined by some sublattice.  Its proof will be published in \cite{LZp}.
	
	\begin{theorem}\label{colcoinequiv}
		Let $\G_2$ be a sublattice of $\G_1$ of index $m$, with $\G_1=\bigcup_{i=1}^m(c_i+\G_2)$ where $c_1=0$. Then $R\in OC(\G_1)$ is a colour coincidence of the colouring of $\G_1$ 
		determined by $\G_2$ if and only if $R$ fixes colour $c_1$.
	\end{theorem}
	
	Hence, it is sufficient to verify if $R[\G_2\cap\G_1(R^{-1})]=\G_2\cap\G_1(R)$ is satisfied to conclude whether $R\in OC(\G_1)$ is a colour coincidence of the colouring of the lattice 
	$\G_1$ induced by the sublattice $\G_2$ of $\G_1$.  Furthermore, because of Theorem \ref{colcoinequiv}, a colour coincidence $R$ actually defines a permutation on the set of cosets of 
	$\G_2(R)$ in $\G_1(R)$ with $Rc_i+\G_2(R)=c_j+\G_2(R)$ if $R$ sends colour $c_i$ to $c_j$.

	The next result links the property that a coincidence isometry $R$ of $\G_1$ is a colour coincidence of the colouring of $\G_1$ determined by $\G_2$ with the relationship between 
	$\S_1(R)$ and $\S_2(R)$.
	
	\begin{corollary}\label{divandeq}
		Let $\G_1$ be a lattice having $\G_2$ as a sublattice of index $m$.
		\begin{enumerate}
			\item If $R$ is a colour coincidence of the colouring of $\G_1$ determined by $\G_2$ then $\S_2(R)\mid\S_1(R)$.

			\item If $s=t=m$, then $R$ is a colour coincidence of the colouring of $\G_1$ determined by $\G_2$ if and only if $\S_2(R)=\S_1(R)$.
		\end{enumerate}
	\end{corollary}
	\begin{proof}
		\begin{enumerate}
			\item By Theorem \ref{colcoinequiv} and Lemma \ref{intandeqthm}, $R[\G_2\cap\G_1(R^{-1})]=\G_2\cap\G_1(R)=\G_2(R)$.  Thus, $u=v=1$ and by \eqref{s2r}, 
			$\S_1(R)=\frac{m}{t}\cdot\S_2(R)$.  Since $t\mid m$, $\S_2(R)\mid\S_1(R)$.

			\item It follows from above that $\S_2(R)=\S_1(R)$ whenever $R$ is a colour coincidence of the colouring of $\G_1$ and $t=m$.  Conversely, if 
			$\S_2(R)=\S_1(R)$ then $[\G_1(R):\G_2(R)]=t\cdot u=s\cdot v=m$ (see Figure \ref{latticediag}) and thus, $u=v=1$.  Hence, $R[\G_2\cap\G_1(R^{-1})]=\G_2\cap\G_1(R)$ or $R$ 
			fixes colour $c_1$. By Theorem \ref{colcoinequiv}, $R$ is a colour coincidence of the colouring of $\G_1$.
		\end{enumerate}
	\end{proof}

	The first result is a reformulation of a remark in \cite{B97}.  The condition $s=t=m$ in the second result means that all $m$ colours of the colouring of $\G_1$ appear in the colourings 
	of $\G_1(R^{-1})$ and $\G_1(R)$ determined by $\G_2$.

	\begin{example}
		Take the rectangular sublattice $\G_2$ with basis $\set{[3,0]^T,[0,1]^T}$ that is of index 3 in the square lattice $\G_1=\mathbbm{Z}^2$, and the coincidence rotation $R$ of 
		Example \ref{ex1}.  Both colourings of $\G_1(R^{-1})$ and $\G_1(R)$ include three colours (see Figure \ref{extwo}), and thus $s=t=3$.  Observe that $R$ fixes the colour black 
		(corresponding to the coset $c_1+\G_2=\G_2$), which means that $R$ is a colour coincidence of the colouring of $\G_1$ by Theorem \ref{colcoinequiv}.  Indeed, $R$ fixes the 
		colour black, and interchanges the colours blue and red.  Moreover, by Corollary \ref{divandeq}, $\S_2(R)=\S_1(R)$.  The CSL $\G_2(R)$ is made up of all black points in the 
		colouring of $\G_1(R)$.
		\begin{figure}
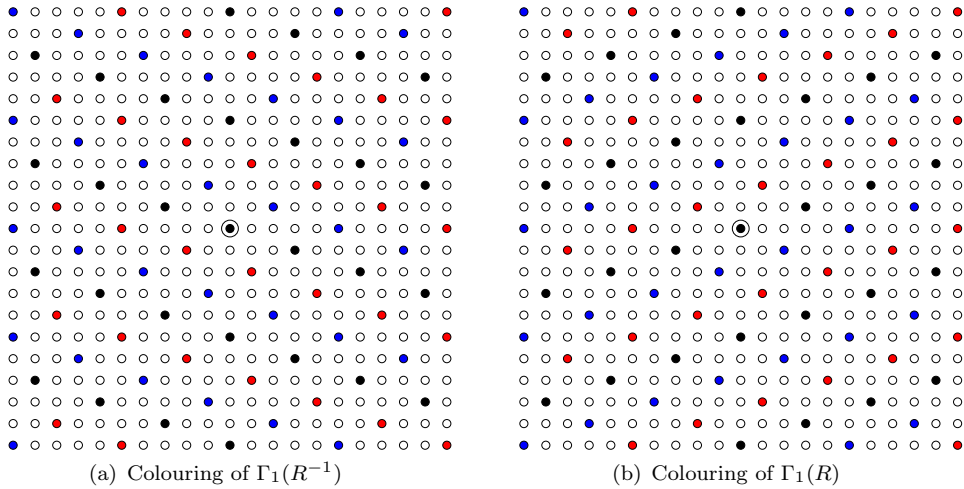

			\begin{center}	
				\subfigure[Colouring of $\G_1(R^{-1})$]{
				\input{extwoinvb.tex}				
				}
				\subfigure[Colouring of $\G_1(R)$]{
				\input{extwob.tex}
				}
				\caption{The coincidence rotation $R$ (the same as in Figure \ref{exone}) is a colour coincidence of the colouring of $\G_1=\mathbbm{Z}^2$ determined by the 
				sublattice $\G_2$ generated by $[3,0]^T$ and $[0,1]^T$ because the colour black (corresponding to $\G_2$) is fixed by $R$.  The blue and red colours are 
				interchanged by $R$.  The CSL $\G_2(R)$ consists of the black points in the colouring of $\G_1(R)$. (Online: coloured version)}\label{extwo}
			\end{center}			
		\end{figure}
	\end{example}

\section{Example on the Ammann-Beenker tiling}
	Even though all of the previous results were stated for lattices, they also hold for $\mathbbm{Z}$-modules in $\mathbbm{R}^d$. \cite{LZp}  This suggests that the same techniques are 
	applicable when dealing with quasicrystals.  In fact, the coincidence problem for the set of vertex points of a quasicrystalline tiling breaks into two parts: the coincidence problem 
	for the underlying limit translation module, and the computation of the window correction factor. \cite{PBR96,B97}  The latter depends on the geometry of the window and equals 
	1 if the window is circular.  In such a case, results on the underlying module are exactly the same as that on the set of vertex points of the tiling.

	To illustrate, we consider the set of vertex points $P_1$ of the eightfold symmetric Ammann-Beenker tiling (see Figure \ref{abtsubset}).  We now have to consider the 
	coincidence problem of the underlying module $M_1$ which is the standard eight-fold planar module $\mathbbm{Z}[\xi]$ of rank 4, where $\xi=e^{\pi i/4}$.  That is, we treat $P_1$ as a 
	discrete subset of $M_1$ embedded in the complex plane.  Due to lack of space, we omit the discussion of the window correction factor here, but notice that the acceptance factor for the 
	Ammann-Beenker tiling is anyway very close to one. \cite{PBR96}
	\begin{figure}
		\begin{center}
			\input{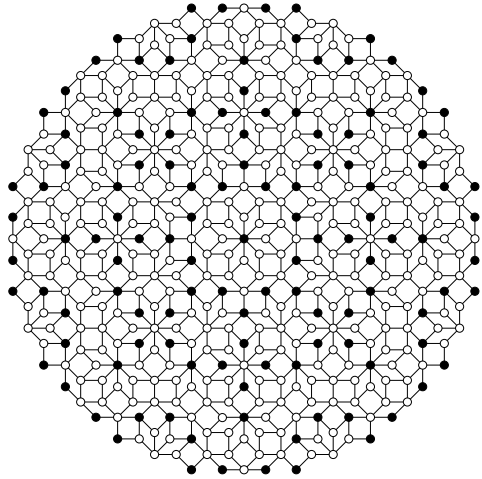}	
		\end{center}
		\caption{The set of vertices $P_1$ of the Ammann-Beenker tiling and the subset $P_2$ (coloured black) of $P_1$}\label{abtsubset}
	\end{figure}  
	
	Let $P_2$ be the set of points of $P_1$ coloured black in Figure \ref{abtsubset}. The subset $P_2$ can be obtained as the intersection of the submodule $M_2=\langle 
	1+\xi^2\rangle$ (of index 4 in $M_1$) and $P_1$.  Choose the coincidence rotation $R$ of $M_1$ (and $P_1$) to be the counterclockwise rotation about the origin at an angle 
	of $\tan^{-1}(-2\sqrt{2})\approx 109.5^{\circ}$.  The coincidence index of $R$ with respect to $M_1$ is 9.
	
	A colouring of $P_1$ with four colours determined by $M_2$ is shown in Figure 1 of \cite{BDLPEFF08}.  Denote by $P_1(R^{-1})$ and $P_1(R)$ the set of coinciding points of $P_1$ and 
	$R^{-1}P_1$, and of $P_1$ and $RP_1$, respectively.  Figures \ref{abtinv} and \ref{abt} show the colourings of $P_1(R^{-1})$ and $P_1(R)$ induced by the colouring of $P_1$ determined by 
	$M_2$.  Observe that all four colours appear in both colourings, and $R$ is a colour coincidence of the colouring (with all colours being fixed).  Hence, the coincidence index of $R$ 
	with respect to $M_2$ is also 9.  Also, the set of coinciding points, $P_2(R)$, consists of all points of $P_2$ in the colouring of $P_1(R)$. 

	\begin{figure}
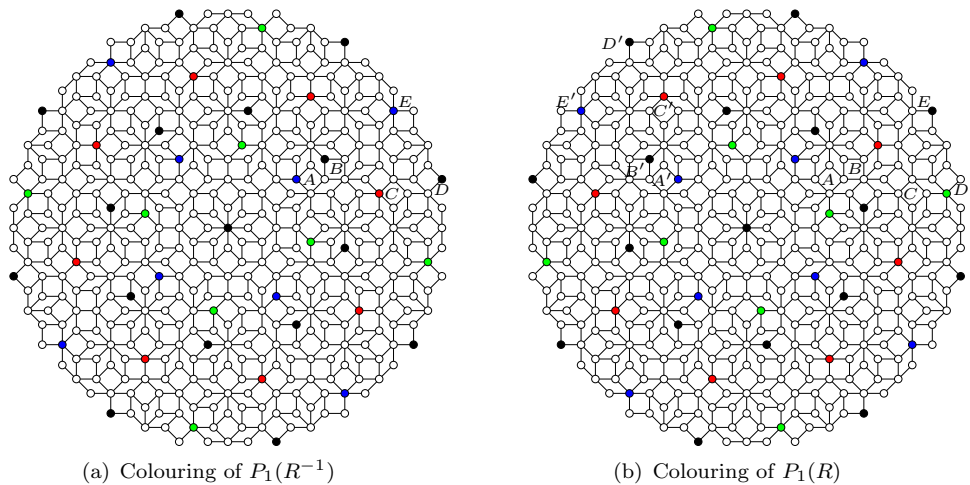

		\begin{center}
			\subfigure[Colouring of $P_1(R^{-1})$]{\label{abtinv}
			\input{abtinv.tex}				
			}
			\subfigure[Colouring of $P_1(R)$]{\label{abt}
			\input{abt.tex}
			}
			\caption{Colourings of $P_1(R^{-1})$ and $P_1(R)$ determined by $P_2$ (see Figure \ref{abtsubset}), where the coincidence rotation $R$ corresponds to a rotation about 
			the origin by $\tan^{-1}(-2\sqrt{2})\approx 109.5^{\circ}$ in the counterclockwise direction.  The points labelled $A$, $B$, $C$, $D$, and $E$ are mapped by $R$ to the 
			points 	labelled $A'$, $B'$, $C'$, $D'$, and $E'$, respectively. $R$ is a colour coincidence of the colouring of $\G_1$ that fixes all the colours. The intersection of 
			$P_2$ and $RP_2$, $P_2(R)$, is made up of all black points in the colouring of $P_1(R)$. (Online: coloured version)}\label{exthree}
		\end{center}			
	\end{figure}

\section{Conclusion and outlook}
	A method of calculating the coincidence index of a coincidence isometry of a lattice with respect to a sublattice is formulated in this paper by considering the colouring of the lattice 
	determined by the sublattice.  Examples in the plane were given, for which the actual colourings were shown.  However, this does not mean that	the results acquired are only applicable 
	when actual colourings of the lattices can be visualised.  Examples for lattices-sublattices in higher dimensions are also possible, as well as general results (for instance, 
	sublattices of prime index) in arbitrary dimension. \cite{LZp}  With these results, it is foreseeable that the coincidence indices of other type of lattices may be computed, as long as 
	the lattices can be treated as sublattices of other lattices for which the coincidence problem has already been completely solved.  Though more involved, these results are still 
	applicable in the quasicrystal case.
	
	The generalisation of a colour symmetry to colour coincidences is not only interesting in its own right, but also provides a further connection between the relationship of the 
	coincidence indices of a lattice and sublattice.  Moreover, Theorem \ref{colcoinequiv} permits a simpler way of determining the colour symmetry group of colourings of lattices and 
	$\mathbbm{Z}$-modules.

\section*{Acknowledgements}
	M. Loquias would like to thank the Deutscher Akademischer Austausch Dienst (DAAD) for financial support during his stay in Germany. This work was supported by 
	the German Research Council (DFG), within the CRC 701.

\end{document}